\documentclass[11pt]{amsart}

\usepackage{latexsym, amssymb,enumerate}
\usepackage{appendix}

\newcommand{\be}{\begin{equation}}
\newcommand{\ee}{\end{equation}}
\newcommand{\beq}{\begin{eqnarray}}
\newcommand{\eeq}{\end{eqnarray}}

\def\H{{\mathbb H}}

\newtheorem{prop}{Proposition}[section]
\newtheorem{thm}[prop]{Theorem}
\newtheorem{lemm}[prop]{Lemma}

\newtheorem{rema}[prop]{Remark}

\def\begeq{\begin{equation}}
\def\endeq{\end{equation}}

\def\R{\mathbb R}

\def\s{\sigma}
\def\l{\lambda}

\def\S{{\mathbb S}}

\def \ds{\displaystyle}
\def \vs{\vspace*{0.1cm}}

\def\odot{\setbox0=\hbox{$\bigcirc$}\relax \mathbin {\hbox
to0pt{\raise.5pt\hbox to\wd0{\hfil $\wedge$\hfil}\hss}\box0 }}

\numberwithin{equation} {section}

\begin{document}

\title[Hyperbolic Alexandrov-Fenchel inequalities. I] {Hyperbolic Alexandrov-Fenchel quermassintegral inequalities. I}

\author{Yuxin Ge}
\address{Laboratoire d'Analyse et de Math\'ematiques Appliqu\'ees,
CNRS UMR 8050,
D\'epartement de Math\'ematiques,
Universit\'e Paris Est-Cr\'eteil Val de Marne, \\61 avenue du G\'en\'eral de Gaulle,
94010 Cr\'eteil Cedex, France}
\email{ge@u-pec.fr}
\author{Guofang Wang}
\address{ Albert-Ludwigs-Universit\"at Freiburg,
Mathematisches Institut
Eckerstr. 1
D-79104 Freiburg}
\email{guofang.wang@math.uni-freiburg.de}

\author{Jie Wu}
\address{School of Mathematical Sciences, University of Science and Technology
of China Hefei 230026, P. R. China
\and
 Albert-Ludwigs-Universit\"at Freiburg,
Mathematisches Institut
Eckerstr. 1
D-79104 Freiburg
}
\email{jie.wu@math.uni-freiburg.de}

\thanks{The first named author  is partly supported by ANR  project
ANR-08-BLAN-0335-01. The  second and third named authors are partly supported by SFB/TR71
``Geometric partial differential equations''  of DFG}
\begin{abstract}
In this paper we prove the following geometric inequality in the hyperbolic space $\H^n$ ($n\ge 5)$, which is a hyperbolic  Alexandrov-Fenchel inequality,
\[\begin{array}{rcl}
\ds \int_\Sigma \s_4 d \mu\ge \ds\vs C_{n-1}^4\omega_{n-1}\left\{ \left( \frac{|\Sigma|}{\omega_{n-1}} \right)^\frac 12 +
\left( \frac{|\Sigma|}{\omega_{n-1}} \right)^{\frac 12\frac {n-5}{n-1}} \right\}^2,
\end{array}\]
provided that $\Sigma$ is a horospherical convex hypersurface. Equality holds if and only if $\Sigma$ is a geodesic sphere in $\H^n$.
\end{abstract}

\maketitle

\section{Introduction}
The Alexandrov-Fenchel inequalities for quermassintegrals  of convex domains in $\R^n$, as a generalization of the classical  isoperimetric inequality,
play an important role in classical geometry. For a bounded smooth domain $\Omega \subset \R^n$  with boundary $\partial\Omega =\Sigma$, let $\kappa=(\kappa_1,\kappa_2,\cdots, \kappa_{n-1})$ be the set of the principal curvatures  of $\Sigma$ and $\s_k:\R^{n-1}\to \R$ the $k$-th  elementary symmetric function.
One of equivalent definitions of the quermassintegrals  of $\Omega$ is
\begin{equation}\label{quermass}
V_{n-k}(\Omega)=c_{n,k}\int_\Sigma\s_{k-1}(\kappa), \quad k\ge 1,
\end{equation}
where $c_{n,k}=C_{n-1}^k/C_{n-1}^{k-1}$. In this paper   we denote $C_{n-1}^{k}=\frac{(n-1)!}{k!(n-1-k)!}$. $V_n$ is the volume of $\Omega$ up to a  constant multiple and $V_{n-1}$ the area of $\Sigma$.
The celebrated Alexandrov-Fenchel quermassintegral inequalities state that if $\Omega$ is convex, then for $0\le i<j<n$,
\begin{equation}\label{eq01}
\frac{ V_{n-j}(\Omega)^{\frac 1{n-j}}}{ V_{n-j}(B)^{\frac 1{n-j}}}\ge
\frac{ V_{n-i}(\Omega)^{\frac 1{n-i}}}{ V_{n-i}(B)^{\frac 1{n-i}}}.
\end{equation}
 When $i=0$ and $j=1$,
\eqref{eq01} is the isoperimetric inequality. For $k\geq 1$ in (\ref{quermass}),
inequality (\ref{eq01})
 is equivalent to
\begin{equation}\label{eq01_1}
\int_{\Sigma}\s_k \ge C_{n-1}^k \omega_{n-1} \left(  \frac 1{C_{n-1}^j}\frac 1{\omega_{n-1}}\int_{\Sigma}\s_{j} \right)^{\frac {n-1-k}{n-1-j}}, \quad 0\le j< k\le  n-1,
\end{equation}
where $\omega_{n-1}$ is the area of the standard sphere $\S^{n-1}$.
(In this paper we use a convention that $\s_0=1$.)  Hence inequality \eqref{eq01_1} is also interpreted as a generalization of the isoperimetric inequality.
When $k=1$, this is usually called a Minkowski inequality. When $j=0$, it is
\begin{equation}\label{eq02}
\int_{\Sigma}\s_k \ge C_{n-1}^k \omega_{n-1} \left( \frac {|\Sigma|}{\omega_{n-1}}\right)^{\frac{n-1-k}{n-1}}.
\end{equation}
 The Alexandrov-Fenchel quermassintegral inequalities in $\R^n$ have been intensively
studied in the last several decades. See classical books of Santos\cite{Santos}, Burago-Zalgaller \cite{BuragoZalgaller} and Schneider \cite{Schneider}.
 For the non-convex domains, see the recent interesting work of Guan-Li \cite{GuanLi} and Huisken \cite{Huisken} and also the work of Chang-Wang \cite{ChangWang}.
 Here we just mention a conformal version of  Alexandrov-Fenchel quermassintegral inequalities in $\S^n$ in \cite{GuanWang}, which is curious that there is a very closed relation to
 the hyperbolic Alexandrov-Fenchel quermassintegral inequalities which we want to establish in this paper.

In this paper we are interested in its analogue in the  hyperbolic space $\H^n$. Let $\H^n=\R^+\times \S^{n-1}$ with the hyperbolic metric
\[\bar g=dr^2+\sinh ^2 r g_{\S^{n-1}},\]
where $g_{\S^{n-1}}$ is the standard round metric on the unit sphere $\S^{n-1}$.
The isoperimetric inequality on the  hyperbolic space $\H^n$ was obtained by Schmidt \cite{Schmidt}. See a flow approach
in \cite{Makowski}. When $n=2$, the  hyperbolic isoperimetric inequality
is
\[ L^2\ge 4\pi A+A^2,\]
where $L$ is the length of a  curve $\gamma$ in $\H^2$ and $A$ is the area of the enclosed domain by $\gamma$. Moreover, equality holds if and only if $\gamma$ is a circle.
 There are many attempts to establish  Alexandrov-Fenchel inequalities on the
hyperbolic space $\H^n$. See,  for example, \cite{Schlenker} and \cite{Solanes_Thesis}. In \cite{GS}, Gallego-Solanes proved by using integral geometry
the following interesting inequality for convex domains in $\H^n$, precisely, there holds,
\be \label{gs}
\int_\Sigma \s_k d\mu  > cC_{n-1}^k |\Sigma|,\ee
where $c=1$ if $k>1$ and $c=(n-2)/(n-1)$ if $k=1$ and $|\Sigma|$ is the area of $\Sigma$. Here $d\mu$ is the area element of the induced metric. See also \cite{BM}.
The above inequality \eqref{gs} ($k>1$) is sharp in the sense that the constant $c$ could not be improved.
However, this inequality is far away from being optimal, especially when $|\Sigma|$ is small.

Recently motivated by the study of the quasi-local mass and the Penrose inequality, Brendle-Hung-Wang \cite{BHW}
established the following Minkowski type inequalities  (i.e., $k=1$)
\begin{equation}\label{eq03}
    \int_{\Sigma}\bigg(\lambda'H-(n-1)\langle\bar{\nabla}\lambda',\nu\rangle\bigg)d\mu\geq(n-1)\omega_{n-1}^{\frac 1{n-1}}{|\Sigma|}^{\frac{n-2}{n-1}},
\end{equation}
and  de Lima and Girao \cite{deLG}  proved the following related inequality
\begin{equation}\label{eq04}
    \int_{\Sigma}\lambda'Hd\mu\geq (n-1)\omega_{n-1}\left(\big(\frac {|\Sigma|}{\omega_{n-1}})^{\frac{n-2}{n-1}}+(\frac {|\Sigma|}{\omega_{n-1}}\big)^{\frac{n}{n-1}}\right),
\end{equation}
where  $\lambda'(r)=\cosh r$, if $\Sigma$ is star-shaped and mean convex (i.e. $H>0$). This  total mean curvature integral with the weight  $\lambda' $ appears naturally in
the definition of the Brown-York mass and the Liu-Yau mass
and in the Penrose inequality for asymptotically  hyperbolic graphs \cite{DGS}.
The higher order mean curvature integrals  $\int \lambda'\s_{2k-1}$ appear also in our work \cite{GWW2} on a new mass on asymptotically  hyperbolic graphs. It is
 an interesting open problem if
one can generalize \eqref{eq03} and \eqref{eq04} for general $k$.

In this paper we are interested in the Alexandrov-Fenchel quermassintegral inequalities in $\H^n$ for curvature integrals without the weight $\lambda'$, i.e,
for
\[\int_\Sigma\sigma_k d\mu.\]
Such curvature integrals, like in the Euclidean case, have a close relationship with quermassintegral in $\H^n$. See for example \cite{Santos, GS, Solanes_Thesis}.
We will give more details in our forthcoming paper \cite{GWW_AF_k}.
Such an inequality for $\s_2$ was first proved by  Li-Wei-Xiong in a recent work \cite{LWX}.
\begin{align}\label{eq05}
    \int_{\Sigma}\sigma_2d\mu\geq&\frac{(n-1)(n-2)}2~\left(|\Sigma|+\omega_{n-1}^{\frac{2}{n-1}}{|\Sigma|}^{\frac{n-3}{n-1}}\right),
\end{align}
provided that $\Sigma\subset\H^n$ is a star-shaped and two-convex hypersurface, ie., $\sigma_1\ge 0$ and $\sigma_2\ge 0$.

In this paper we obtain

\begin{thm}
\label{thm_AF2} Let $n\ge 5$. If $\Sigma\subset\H^n$ is horospherical convex, then
\begin{equation}\label{AF2}
\begin{array}{rcl}
\ds \int_\Sigma \s_4 d \mu\ge \ds\vs C_{n-1}^4\omega_{n-1}\left\{ \left( \frac{|\Sigma|}{\omega_{n-1}} \right)^\frac 12 +
\left( \frac{|\Sigma|}{\omega_{n-1}} \right)^{\frac 12\frac {n-5}{n-1}} \right\}^2,
\end{array}
\end{equation}where $\omega_{n-1}$ is the area of the unit sphere $\S^{n-1}$ and $|\Sigma|$ is the area of $\Sigma$.
Equality holds if and only if $\Sigma$ is a geodesic sphere. Moreover, when $n=5$, (\ref{AF2}) holds 
provided that $\Sigma\subset\H^n$ is a star-shaped and two-convex hypersurface, ie., $\sigma_1\ge 0$ and $\sigma_2\ge 0$.
\end{thm}

Theorem \ref{thm_AF2} implies trivially an (equivalent) isoperimetric type result: {\it In the class of horospherical convex hypersurfaces with fixed area, the minimum of $\int\s_4 d\mu$
is achieved by and only by geodesic spheres.}

Unlike in the case of $\R^n$,
 we believe that in the case of $\H^n$ the parity of $k$ in the $k$-scalar curvature $\s_k$ plays a role in the Alexandrov-Fenchel quermassintegral inequalities\footnote{Haizhong Li has
 a similar idea \cite{LiHaizhong}}.
 Namely, for odd $k$ the Alexandrov-Fenchel quermassintegral inequalities should look
  like \eqref{eq03} or \eqref{eq04}, while for even $k$ \eqref{eq05} and \eqref{AF2} are the correct and the best ones. See also Theorem \ref{thm2} below.
 In a forthcoming paper \cite{GWW_AF_k} we will establish the Alexandrov-Fenchel quermassintegral inequalities for general even $k$.

 $\Sigma\subset \H^n$ is
 {\it horospherical convex}  if all principal curvatures are larger than or equal to $1$. The horospherical convexity is a natural geometric concept, which is equivalent to
 the geometric convexity in Riemannian manifolds. Through  our work, we believe that it is (almost) the
 best condition for inequality \eqref{AF2}. See Remark \ref{rk1} and Remark \ref{rk2} below.

The fundamental idea to show the above geometric inequalities is the same: Consider a suitable functional and  a suitable geometric flow  and prove this
functional is non-increasing under the  geometric flow. If the flow converges to the standard sphere, then we have an inequality, with a best constant achieved by the standard sphere.
 The flow we use is the inverse curvature flow
studied by Gerhardt \cite{Gerhardt}
\begin{equation}\label{flow}
\frac{\partial X}{\partial t}=\frac{n-4}4\frac {\s_3}{\s_4}\nu,
\end{equation}
where $\nu$ is the outer normal of $\Sigma$.

The first problem we meet is: what is the  suitable functional for our inequality (\ref{AF2})? By the work of Brendle-Hung-Wang \cite{BHW}, de Lima-Gir\~{a}o \cite{deLG}
and Li-Wei-Xiong \cite{LWX}, one may guess that the following
functional
\begin{equation}\label{func}
Q(\Sigma):=|\Sigma|^{-\frac{n-5}{n-1}} \int_{\Sigma_t}\bigg\{\sigma_4-\frac{(n-3)(n-4)}{6}\sigma_2+\frac{(n-1)(n-2)(n-3)(n-4)}{24}\bigg\},
\end{equation}
could be a good candidate. In fact,
\[l_2:=\bigg(\sigma_4-\frac{(n-3)(n-4)}{6}\sigma_2+\frac{(n-1)(n-2)(n-3)(n-4)}{24}\bigg),\]
is the Gauss-Bonnet curvature $L_2$ up to a  constant multiple. The Gauss-Bonnet curvature $L_2$
is an intrinsic geometric invariant which is a natural generalization of the scalar curvature $R$. For the Gauss-Bonnet curvature $L_2$, see for example \cite{GWW}.
We remark that the functional considered in \cite{LWX} is in fact
the Yamabe quotient for the scalar curvature $R$.
Our functional \eqref{func} is also a Yamabe type quotient for $L_2$.
Due to the complication of the geometry of the hyperbolic space, we obtain a variation formula of $\int l_2$, which has three terms
that we have to deal with. Unlike  the cases of proving inequalities  \eqref{eq03}, \eqref{eq04} and \eqref{eq05}, one can not directly use the Newton-MacLaurin inequalities
(Lemma \ref{lem}) to deal with these three terms. More precisely, among the three terms appeared in (\ref{L2evolve}),
\[ 5\frac{\sigma_5\sigma_3}{\sigma_4}-\frac{4(n-5)}{n-4}\sigma_4,
 \quad \frac{4(n-3)}{n-4}\sigma_2-3\frac{\sigma_3^2}{\sigma_4},\]
 are non-positive, and the term
 \[
\frac{\sigma_1\sigma_3}{\sigma_4}-\frac{4(n-1)}{n-4},\]
is non-negative  in the use of the  Newton-MacLaurin inequalities. For the monotonicity of the functional $Q$ we need to show that
the sum of these 3 term is non-positive.
Hence we have to deal with them together. Since these terms have different scaling of $\kappa$, one could not expect the sum of these three terms is non-positive for
all $\kappa =(\kappa_1,\cdots,\kappa_{n-1})\in \R^{n-1}_+$. Fortunately we show that it does be non-positive, if $\kappa_i\ge 1 $ for all $i$. See Proposition \ref{keyprop.}.
This is one of crucial points of this paper, where the assumption of the horospherical convexity plays a crucial role. By applying the work of Gerhardt on the inverse curvature flow, one can show that
flow (\ref{flow}) preserves the  condition of the horospherical convexity. Therefore, the functional $Q$ defined in (\ref{func}) is non-increasing under  flow (\ref{flow}).
Hence, in order to obtain an inequality we now only need to consider the limit of $Q$ under  flow \eqref{flow}. Now, we meet another problem,
flow \eqref{flow} converges only asymptotically in the sense presented in Proposition 2.2, namely the flow converges asymptotically  to a sphere with a metric $g$ conformal to
the standard round metric on $\S^{n-1}$. We show that along the flow, the induced metric has  (asymptotically) positive Schouten tensor, if the evolving hypersurface is horospherical convex.
For such a metric on $\S^{n-1}$, a generalized Sobolev
inequality was proved by Guan-Wang \cite{GuanWang} in conformal geometry. Here the horospherical convexity plays again an important role.
With this inequality we get a best estimate for $Q$ in Theorem \ref{thm2}.  Therefore Theorem \ref{thm_AF2} follows.
It is interesting to see that the results in conformal geometry on
the sphere are closely related to the hyperbolic Alexandrov-Fenchel inequality in $\H^n$. The connecting bridge  is the inverse curvature flow of Gerhardt \cite{Gerhardt}. See also the work of Ding \cite{Ding}. For the recent related work see \cite{B, BM, KM}.

The rest of this paper is organized as follows. In Section 2 we present some basic facts about the elementary functions $\sigma_k$, the variational formula for $\int\sigma_k$
and recall a generalized Sobolev inequality from \cite{GuanWang}. The preservation of the  horospherical convexity under a inverse curvature flow  considered by Gerhardt,
together with its convergence, is given in this section. In Section 3, we prove the crucial monotonicity of $Q$, analyze its asymptotic behavior under flow \eqref{flow}, and
prove Theorem \ref{thm_AF2}.

\section{Preliminaries}

Let $\s_k$ be the $k$-th elementary symmetry function $\s_k:\R^{n-1}\to \R$ defined by
\[\s_k(\Lambda)=\sum_{i_1<\cdots<i_{k}}\l_{i_1}\cdots\lambda_{i_k}\quad  \hbox{ for } \Lambda=(\l_1, \cdots,\l_{n-1})\in \R^{n-1}.\]
The definition of $\s_k$ can be easily  extended to the set of all symmetric matrix. The Garding cone $\Gamma_k^+$ is defined as
\[\Gamma_k^+=\{\Lambda \in \R^{n-1} \, |\,\s_j(\Lambda)>0, \quad\forall j\le k\}.\]
We collect the basic facts about $\s_k$, which will be directly used in this paper. For other related facts, see a survey of Guan \cite{Guan} or \cite{LWX}.
\begin{lemm} \label{lem} For $\Lambda\in\Gamma_k^+$, we have the following Newton-MacLaurin inequalities
\begin{align}
    &\frac{\sigma_{k-1}\sigma_{k+1}}{\sigma_k^2}\leq \frac{k(n-k-1)}{(k+1)(n-k)}\label{eq21},\\
    &\frac{\sigma_1\sigma_{k-1}}{\sigma_k}\geq \frac{k(n-1)}{n-k}.\label{eq22}
  \end{align}
Moreover,  equality holds in \eqref{eq21} or \eqref{eq22} at $\Lambda$ if and only if $\Lambda=c(1,1,\cdots,1)$.
\end{lemm}
The Newton-MacLaurin inequalities play a very important role in proving geometric inequalities mentioned above. However, we will see that these inequalities are not precise enough to show our
inequality (\ref{AF2}).

 Let $\H^n=\R^+\times \S^{n-1}$ with the hyperbolic metric
\[\bar g=dr^2+\sinh ^2 r g_{\S^{n-1}},\]
where $g_{\S^{n-1}}$ is the standard round metric on the unit sphere $\S^{n-1}$
and $\Sigma \subset \H^n$ a smooth closed hypersurface in $\H^n$ with a unit outward normal $\nu$.
Let $h$ be the second fundamental form of $\Sigma$ and $\kappa=(\kappa_1,\cdots,\kappa_{n-1})$  the set of principal curvatures of $\Sigma$ in $\H^n$
with respect to $\nu$.
The $k$-th mean curvature
of $\Sigma$ is defined by
\[\s_k=\s_k(\kappa).\]

We now  consider  the following curvature evolution equation
\begin{equation}\label{flow_g}
    \frac{d}{dt} X=F\nu,
\end{equation}
where $\Sigma_t=X(t,\cdot)$ is a family of hypersurfaces in $\H^n$,
 $\nu$ is the unit outward normal to $\Sigma_t=X(t,\cdot)$ and $F$ is a speed function which may depend on the position vector $X$ and  principal curvatures of $\Sigma_t$.  One can check that
 along  flow (\ref{flow_g}),
 \begin{align}\label{var}
    \frac d{dt}\int_{\Sigma}\sigma_k d\mu=&(k+1)\int_{\Sigma}F\sigma_{k+1}d\mu+(n-k)\int_{\Sigma}F\sigma_{k-1}d\mu.
\end{align}
For a proof see for instance \cite{Reilly}. Here we use a convention $\sigma_{-1}=0$.
 If one compares  flow (\ref{flow_g}) in $\H^n$ with a similar flow of hypersurfaces in $\R^n$, the last term in (\ref{var}) is an extra term. This extra term comes from  the  sectional curvature $-1$ of $\H^n$ and makes the phenomenon of hypersurfaces in  $\H^n$ much different from the one of hypersurfaces  in $\R^n$.

 As mentioned above, we use exactly the following inverse flow
 \begin{equation}\label{flow1}
    \frac{d}{dt}X=\frac{n-4}{4}\frac{\sigma_3}{\sigma_4}\nu.
\end{equation}

By using the result of Gerhardt \cite{Gerhardt}, we have
\begin{prop}
If the initial hypersurface $\Sigma$ is horospherical convex, then the solution for flow \eqref{flow1} exists for all time $t>0$ and preserves the condition of horospherical convexity. Moreover, the hypersurfaces $\Sigma_t$ become  more and more  umbilical in the sense of
\begin{equation*}
    |h^i_j-\delta^i_j|\leq Ce^{-\frac t{n-1}},\quad t>0,
\end{equation*}
i.e., the principal curvatures are uniformly bounded and converge exponentially fast to one. Here
 $h^i_j=g^{ik}h_{kj}$, where $g$ is the induced metric and $h$ is the second fundamental form.
 \end{prop}
\begin{proof}
In \cite{Gerhardt} Gerhardt studied a more general inverse flow under a
weaker condition that the initial surface is star-shaped.
\begin{equation}\label{flow_gen} \frac d{dt}X=-\Phi(F)\nu,\end{equation}
with a function $\Phi(r)=-r^{-1}$ for $r>0$ and $F$ is a smooth curvature function, homogeneous of degree $1$, monotone, and concave.
 What we only need to check is that
flow (\ref{flow1})  or the general flow \eqref{flow_gen} preserves the condition of horospherical convexity.

By using  (4.23) in \cite{Gerhardt} for the second fundamental form $h_j^i$,
we have the evolution equation for $\tilde h_j^i:=h_j^i-\delta_j^i$ that
\begin{equation}\label{eq_a1}
\begin{array}{rcl}
\ds \dot {\tilde h}^i_j &=& \ds\vs Q(\nabla^2\tilde h, \nabla \tilde h)^i_j+\dot\Phi F^{kl}h_{lr}h^r_k\tilde h_j^i+ \dot\Phi F^{kl}h_{lr}h^r_k \delta_j^i \\
&& \ds\vs +(\Phi-\dot\Phi F) \tilde h^{ik}\tilde h_{kj}+2(\Phi-\dot\Phi F) \tilde h^i_j+(\Phi-\dot\Phi F) \delta^i_j\\
&& \ds\vs -\{(\Phi+\dot\Phi F)\delta^i_j-\dot\Phi F^{kl}g_{kl}\tilde h^i_j-\dot\Phi F^{kl}g_{kl}\delta_j^i\}.\\
&:=&\ds\vs Q(\nabla^2\tilde h, \nabla \tilde h)^i_j+H^i_j,
\end{array}\end{equation}
where
$$
\begin{array}{rcl}
Q(\nabla^2\tilde h, \nabla \tilde h)^i_j &=& \ds\vs
 \dot\Phi F^{kl} \tilde h^i_{j;kl}+\dot\Phi F^{kl,rs} \tilde h_{kl;j} {{ \tilde h_{rs;} }}^{\quad i}+\ddot \Phi(F^{kl}\tilde h_{kl;j})(F^{rs}\tilde h_{rs;}^{\quad i})
 \\
 &=:& \dot\Phi F^{kl} \tilde h^i_{j;kl}+ N^i_j.
\end{array}
$$

In order to use the maximum principle for symmetric tensors in \cite{An}, which is a refinement of Hamilton's maximum principle \cite{BN},
 we need to check the following two statements:
\begin{itemize}
\item[(i)] $H^i_ja^ja_i\ge 0,$
\item[(ii)] $N^i_ja_ia^j+\dot \Phi \sup_{\Gamma}2F^{kl}(2\Gamma_k^p\tilde h_{ip;l} a^i-\Gamma_k^p\Gamma_l^q\tilde h_{pq})\ge 0$,\end{itemize}
for any  $a=(a_1,\cdots,a_{n-1}) \hbox{ with } \tilde h_{j}^ia^j=0$,
where $a^j=g^{jl}a_l$.

From \eqref{eq_a1} it is easy to say that
\[ H^i_ja^ja_i= \dot\Phi( F^{kl}h_{lr}h^r_k+F^{kl}g_{kl}-2F)|a|^2 .\]
Now, we can write
\[
F= F^{kl}h_{kl},
\]
for $F$ is homogeneous of degree $1$. Therefore, we infer in the orthonormal basis
\begin{eqnarray*}
&&F^{kl}h_{lr}h^r_k+F^{kl}g_{kl}-2F\\
&=&F^{kl}g_{ls}(\tilde h_r^s+\delta_r^s)(\tilde h_k^r+\delta_k^r)+F^{kl}g_{kl}-2F^{kl}g_{ls}(\tilde h_k^s+\delta_k^s)\\
&=&F^{kl}g_{ls}\tilde h_{r}^s\tilde h^r_k\ge 0,
\end{eqnarray*}
since  $F^{kl}$ is positive  definite. On the other hand, we have
 \[
\dot\Phi>0.\]
Hence, (i) follows. 

Noticing that
\[N^i_j=(\Phi(F))^{kl,rs}\tilde h_{kl;j} {{ \tilde h_{rs;} }}^{\quad i},\]
and that the smallest eigenvalue of $h$ equals to $1$, we can apply the work of Andrews  \cite{An} to show that statement (ii) holds.
Hence the preservation of the horospherical convexity under \eqref{flow1} follows from his
maximum principle.
 \end{proof}
 In \cite{CM}, Cabezas-Rivas and Miquel showed that the preserving volume mean curvature flow preserves the horospherical convexity.
 See also the work of Makowski \cite{Makowski}.

Let $g$ be a Riemannian metric on $M^{n-1}$.
Denote $Ric_g$ and $R_g$  the Ricci tensor and
the scalar curvature of $g$ respectively.  The Schouten tensor is defined
by
\[ A_g=\frac 1{n-3}\left(Ric_g-\frac {R_g}{2(n-2)} g\right).\]
Let $\Lambda_g$ be the set of the eigenvalues of the Schouten tensor $A_g$ with respect to  the metric $g$. The $\s_k$-scalar curvature, which
is introduced by Viaclovsky, is defined by
\[\s_k(g):=\s_k (\Lambda_g).\]
This is a natural generalization of the scalar curvature $R$. In fact, $\s_1(g)=\frac 1{2(n-2)} R$.
Recall that $M$ is of dimension $n-1$. We now consider the conformal class $[g_{\S^{n-1}}]$ of the standard sphere $\S^{n-1}$ and
the following  functionals defined by
\begin{equation}\label{func2}
{ \mathcal F}_k(g)=vol(g)^{-\frac{n-1-2k} {n-1}}\int_{\S^{n-1}} \s_k(g)\, dg, \quad
k=0,1,...,n-1.
\end{equation}
 If a metric $g$ satisfies $\s_j(g)>0$ for all $j\le k$, we call it $k$-positive and denote $g\in \Gamma_k^+$. We recall some generalized Sobolev inequality.

\begin{prop} Let $0 < k < \frac{n-1} 2$ and $g\in [g_{\S^{n-1}}]$ $k$-positive. We have
\begin{equation}
\label{Sk}
{\mathcal F}_k(g)\ge {\mathcal F}_k(g_{\S^{n-1}})=\frac{C_{n-1}^k}{2^k}\omega_{n-1}^{\frac{2k}{n-1}}.
\end{equation}
Moreover, when $k=2$, $n>5$ and $g\in [g_{\S^{n-1}}]$ $1$-positive, the above inequality still holds.
\end{prop}
Inequality (\ref{Sk}) is  a generalized Sobolev inequality, since when $k=1$ inequality \eqref{Sk} is just the optimal Sobolev inequality. See for example \cite{Beckner} and \cite{ChangYang}.
\begin{proof}
The first part follows from Theorem 1.A in \cite{GuanWang}.  When  $k=2$, $n>5$ and $g\in [g_{\S^{n-1}}]$ $1$-positive, by Theorem 1 in \cite{GeWang1} (see also \cite{GuanLinWang}), we infer
\[
(\int_{\S^{n-1}} \s_1(g)\, dg)^{-\frac{n-5} {n-3}}\int_{\S^{n-1}} \s_2(g)\, dg\ge (\frac{n-1}{2})^{-\frac{n-5}{n-3}}\frac{(n-1)(n-2)}{8}\omega_{n-1}^{\frac{2}{n-3}}.
\]
On the other hand, since $g$ is $1$-positive, we have
\[
{\mathcal F}_1(g)\ge {\mathcal F}_1(g_{\S^{n-1}})=\frac{n-1}{2}\omega_{n-1}^{\frac{2}{n-1}}.
\]
Hence, the desired result yields.
\end{proof}

\section{An Alexsandrov-Fenchel inequality in the hyperbolic space}
First, a direct computation gives the following result.
\begin{lemm}
Along the inverse flow (\ref{flow1}), we have
\begin{eqnarray}\label{L2evolve}
&&\frac{d}{dt}\int_{\Sigma}\bigg\{\sigma_4-\frac{(n-3)(n-4)}{6}\sigma_2+\frac{(n-1)(n-2)(n-3)(n-4)}{24}\bigg\}\nonumber\\
&=&(n-5)\int_{\Sigma}\bigg\{\sigma_4-\frac{(n-3)(n-4)}{6}\sigma_2+\frac{(n-1)(n-2)(n-3)(n-4)}{24}\bigg\}\nonumber\\
&&+\frac{n-4}{4}\bigg\{\int_{\Sigma}\bigg(5\frac{\sigma_5\sigma_3}{\sigma_4}-\frac{4(n-5)}{n-4}\sigma_4\bigg)+\frac{(n-4)(n-5)}{6}\bigg(\frac{4(n-3)}{n-4}\sigma_2-3\frac{\sigma_3^2}{\sigma_4}\bigg)\\
&&+\int_{\Sigma}\frac{(n-2)(n-3)(n-4)(n-5)}{24}\bigg(\frac{\sigma_1\sigma_3}{\sigma_4}-\frac{4(n-1)}{n-4}\bigg)\bigg\}.\nonumber
\end{eqnarray}
\end{lemm}

\begin{proof}
Under the flow (\ref{flow_g}), (\ref{var}) yields that
\begin{eqnarray*}
&&\frac{d}{dt}\int_{\Sigma}\bigg\{\sigma_4-\frac{(n-3)(n-4)}{6}\sigma_2+\frac{(n-1)(n-2)(n-3)(n-4)}{24}\bigg\}\\
&=&\int_{\Sigma} 5\sigma_5 F-\frac{(n-4)(n-5)}{2}\sigma_3 F+\frac{(n-2)(n-3)(n-4)(n-5)}{24}\sigma_1 F.
\end{eqnarray*}
Substituting  $F=\frac{n-4}{4}\frac{\sigma_3}{\sigma_4}$ into the previous formula and arranging it,  we thus get the desired result (\ref{L2evolve}).
\end{proof}

\begin{rema}
\label{Remark1}
When $n=5$, Lemma 3.1 implies that $\int_\Sigma l_2$ is a constant. This is acturally the fact that $\int_\Sigma l_2$
is the Euler characteristic of $\Sigma$ up to a constant multiple.

\end{rema}

Compared the last three terms in (\ref{L2evolve}) with the Newton-MacLaurin inequalities (\ref{eq21}),(\ref{eq22}), one will see immediately that
the first two terms are non-positive, but  the last one is non-negative. Therefore, unlike in the papers of \cite{BHW}, \cite{deLG} and \cite{LWX} we can not use the Newton-MacLaurin inequalities directly to establish the desired inequalities. We have to use more precise inequalities, which are
fortunately true for any
\begin{equation}\label{h-convex}
\kappa\in \{\kappa=(\kappa_1,\kappa_2,\cdots,\kappa_{n-1})\in\R^{n-1}\,|\, \kappa_i\ge 1\}.
\end{equation}
This is one of key points of this paper.

\begin{prop}\label{keyprop.}
 Let $n>5$. For any $\kappa$ satisfying (\ref{h-convex}) we have a refined Newton-MacLaurin inequality
\begin{align}\label{p1}\bigg(5\frac{\sigma_5\sigma_3}{\sigma_4}-\frac{4(n-5)}{n-4}\sigma_4\bigg)+\frac{(n-4)(n-5)}{6}\bigg(\frac{4(n-3)}{n-4}\sigma_2-3\frac{\sigma_3^2}{\sigma_4}\bigg)\nonumber\\
+\frac{(n-2)(n-3)(n-4)(n-5)}{24}\bigg(\frac{\sigma_1\sigma_3}{\sigma_4}-\frac{4(n-1)}{n-4}\bigg)\le0.
\end{align}Equality holds if and only if
one of the following two cases holds
\begin{equation}\label{=}
\hbox{either } \quad (i)\,  \kappa_i=\kappa_j \, \forall \, i,j,  \quad\hbox{ or }\quad (ii)\, \exists \, i\, \hbox{ with }\kappa_i >1\, \& \, \kappa_j=1 \, \forall j\neq i.\end{equation}
\end{prop}
\begin{proof} For simplicity of notation, we denote
\begin{align}
p_k=\frac{\sigma_k}{C_{n-1}^k}.
\end{align}
By a direct computation, it is easy to see that \eqref{p1} is equivalent to
\begin{align} \label{p2}
\bigg(\frac{p_5p_3}{p_4}-p_4\bigg)
+2\bigg(p_2-\frac{p_3^2}{p_4}\bigg)
+\bigg(\frac{p_1p_3}{p_4}-1\bigg)\le0.
\end{align}
This inequality follows directly from the following two claims.

\vspace{0.2cm}
\noindent{\bf Claim 1.}  $3(p_2p_4-p_3^2)+(p_3p_1-p_4)\le0.$\quad Equality holds if and only if $\kappa$ satisfies \eqref{=}.
\vspace{0.2cm}

\noindent{\bf Claim 2.}  $3(p_5p_3-p_4^2)+(p_3p_1-p_4)\le0.$\quad Equality holds if and only if $\kappa$ satisfies \eqref{=}.

\

In the proof of these two claims, we replace $n-1$ by $n$ for simplicity of notation
and consider $\kappa=(\kappa_1,\kappa_2,\cdots,\kappa_n)\in \R^n$ with
\[\kappa_i\ge 1, \quad \forall i,\]
and denote $p_k$  the average of
$n$-choose-$k$ type products for $n$ real numbers $\kappa_1,\cdots, \kappa_n$ with  $\kappa_j \ge1$ ($1\le j\le n$).

Let
\begin{align}
F_n(x)&=x^n+C_n^1p_1x^{n-1}+C_n^2p_2x^{n-2}+\cdots+C_n^{n-1}p_{n-1}x+p_n=\Pi_{i=1}^n(x+\kappa_i),
\end{align}
which has exactly $n$ real roots  $-\kappa_i\le-1.$ By the mean value theorem we have that
\begin{align}
\frac1n F_n'(x)=x^{n-1}+\frac{n-1}{n}C_n^1 p_1x^{n-2}+\frac{n-2}{n}C_n^2p_2x^{n-3}+\cdots+\frac{1}{n}C_n^{n-1}p_{n-1}=\Pi_{i=1}^{n-1}(x+\tilde \kappa_i),
\end{align}
is a $(n-1)$-degree
polynomial with $n-1$ real roots $-\tilde k_i \le-1$ ($1\le i\le n-1$). We can write
\begin{align}
\frac1n F_n'(x)=x^{n-1}+C_{n-1}^1p_1x^{n-2}+C_{n-1}^2p_2x^{n-3}+\cdots+C_{n-1}^{n-2}p_{n-2}x+p_{n-1}=\Pi_{i=1}^{n-1}(x+\tilde \kappa_i).
\end{align}
This means that $p_i(1\le i\le n-1)$ of $\kappa\in \R^n$ can be viewed as the average of $(n-1)$-choose-$i$ type products of $\tilde \kappa=(\tilde \kappa_1,\cdots,
\tilde \kappa_{n-1}) \in \R^{n-1}$ with  $\tilde k_j\ge1$ for all $j$.
Namely
\[p_i(\kappa)=p_i(\tilde \kappa), \quad\hbox{ for } 1\le i\le n-1.\]
Therefore, by an argument of mathematical induction, it suffices to prove Claim 1 for $n=4$, and prove Claim 2 for $n=5$.
\vspace{2mm}

\noindent{\it Proof of {Claim 1} for $n=4$}: Given $n$ numbers $(\kappa_1,\kappa_2,\cdots, \kappa_n)$, we denote $\sum\limits_{cyc}f(\kappa_1,\cdots,\kappa_n)$
the cyclic summation which takes over all {\it different} terms of the type $f(\kappa_1,\cdots,\kappa_n)$.
For instance,
\begin{align*}
&\sum_{cyc}\kappa_1=\kappa_1+\kappa_2+\cdots+\kappa_n,\qquad\sum_{cyc}\kappa_1^2\kappa_2=\sum_{i=1}^n\Big(\kappa_i^2\sum_{j\neq i}\kappa_j\Big),\\
&\sum_{cyc}\kappa_1(\kappa_2-\kappa_3)^2=\sum_{i=1}^n\bigg(\kappa_i\sum_{\substack {1\leq j<k\leq n\\j,k\neq i}}(\kappa_j-\kappa_k)^2\bigg),\\
&\qquad\qquad\qquad\quad\;={(n-2)}\sum_{cyc}\kappa_1\kappa_2^2-6\sum_{cyc}\kappa_1\kappa_2\kappa_3.
\end{align*}
When $n=4$, we have
\begin{align*}
p_1=\frac14\sum_{cyc}\kappa_1,\quad p_2=\frac16\sum_{cyc}\kappa_1\kappa_2,\quad p_3=\frac14\sum_{cyc}\kappa_1\kappa_2\kappa_3,\quad p_4=\kappa_1\kappa_2\kappa_3\kappa_4.
\end{align*}
Then we can get
\begin{align}
p_1p_3-p_4=&\frac{1}{16}\Big((\kappa_1+\kappa_2+\kappa_3+\kappa_4)(\kappa_1\kappa_2\kappa_3+\kappa_1\kappa_2\kappa_4+\kappa_1\kappa_3\kappa_4+\kappa_2\kappa_3\kappa_4)-16\kappa_1\kappa_2\kappa_3\kappa_4\Big)\nonumber\\
=&\frac{1}{16}\sum_{cyc}\kappa_1\kappa_2(\kappa_3-\kappa_4)^2,\\
3(p_2p_4-p_3^2)=&3\Big(\frac{1}{6}\kappa_1\kappa_2\kappa_3\kappa_4(\kappa_1\kappa_2+\kappa_1\kappa_3+\kappa_2\kappa_3+\kappa_1\kappa_4+\kappa_2\kappa_4+\kappa_3\kappa_4)\nonumber\\
&-\frac1{16}(\kappa_1\kappa_2\kappa_3+\kappa_1\kappa_2\kappa_4+\kappa_1\kappa_3\kappa_4+\kappa_2\kappa_3\kappa_4)^2\Big)\nonumber\\
=&-\frac{1}{16}\sum_{cyc}\kappa_1^2\kappa_2^2(\kappa_3-\kappa_4)^2,
\end{align}
from above, we can infer that
\begin{align}
3(p_2p_4-p_3^2)+p_1p_3-p_4=\frac{1}{16}\sum_{cyc}\kappa_1\kappa_2(1-\kappa_1\kappa_2)(\kappa_3-\kappa_4)^2\le 0.
\end{align}
This proves Claim 1.\vspace{2mm}

\noindent{\it Proof of {Claim 2} for $n=\noindent5$}: We have:
\begin{align*}
p_1=\frac15\sum_{cyc}\kappa_1,\quad p_3=\frac1{10}\sum_{cyc}\kappa_1\kappa_2\kappa_3,\quad
 p_4=\frac15\sum_{cyc}{\kappa_1\kappa_2\kappa_3\kappa_4},\quad p_5=\kappa_1\kappa_2\kappa_3\kappa_4\kappa_5.
\end{align*}
Hence
\begin{align}
p_1p_3-p_4=&\frac{1}{50}\Big((\kappa_1+\kappa_2+\kappa_3+\kappa_4+\kappa_5)\sum_{cyc}\kappa_1\kappa_2\kappa_3-10\kappa_1\kappa_2\kappa_3\kappa_4\kappa_5\sum_{i=1}^5\frac{1}{\kappa_i}\Big)\nonumber\\
=&\frac{1}{50}\Big(\sum_{cyc}\kappa_1^2\kappa_2\kappa_3-6\sum_{cyc}\kappa_1\kappa_2\kappa_3\kappa_4\Big)\nonumber\\
=&\frac{1}{100}\sum_{cyc}\kappa_1\kappa_2(\kappa_3-\kappa_4)^2,\\
3(p_5p_3-p_4^2)=&3\Big(\frac1{10}\kappa_1\kappa_2\kappa_3\kappa_4\kappa_5\sum_{cyc}\kappa_1\kappa_2\kappa_3-\frac1{25}\big(\sum_{cyc}\kappa_1\kappa_2\kappa_3\kappa_4\big)^2\Big)\nonumber\\
=&\frac{3}{50}\Big(-2\sum_{cyc}(\kappa_1\kappa_2\kappa_3\kappa_4)^2+\kappa_1\kappa_2\kappa_3\kappa_4\kappa_5\sum_{cycle}\kappa_1\kappa_2\kappa_3\Big)\nonumber\\
=&\frac{3}{100}\Big(-\sum_{cyc}(\kappa_1\kappa_2\kappa_3)^2(\kappa_4-\kappa_5)^2\Big).
\end{align}
Therefore, we have
\begin{align}
3(p_5p_3-p_4^2)+(p_3p_1-p_4)=&\frac{1}{100}\sum_{cyc}\kappa_1\kappa_2(\kappa_3-\kappa_4)^2
+\frac{3}{100}\Big(-\sum_{cyc}(\kappa_1\kappa_2\kappa_3)^2(\kappa_4-\kappa_5)^2\Big)\nonumber\\
=&\frac{1}{100}\sum_{cyc}\big[\kappa_1\kappa_2+\kappa_2\kappa_3+\kappa_1\kappa_3-3(\kappa_1\kappa_2\kappa_3)^2\big](\kappa_4-\kappa_5)^2\le0,
\end{align}
which implies Claim 2.  Therefore the proof completes.
\end{proof}

\begin{rema}\label{rk1}
 One can also prove the Proposition by using a replacement with $\kappa_i=1+\tilde\kappa_i$ with $\tilde \kappa_i\ge 0$ for all $i$ and the ordinary Newtow-McLaughlin inequalities
 (Lemma \ref{lem}). With the same proof we present here, Proposition \ref{keyprop.} holds for $\kappa\in \R^{n-1}$ with $\kappa_i\kappa_j\ge 1$ for all $i\neq j$. This is equivalent to
the condition that the sectional curvature of $\Sigma$ is non-negative. We believe that the results proved in this paper for horospherical convex hypersurfaces hold
also for hypersurfaces with non-negative sectional curvature.
\end{rema}

\begin{rema}\label{rk2}
 From the proof of Proposition \ref{keyprop.},  it is easy to see that
\eqref{p1} changes sign for $\kappa\in \R^{n-1}$ with $0\le\kappa_i\le 1$.
\end{rema}

Now we have a monotonicity of $Q(\Sigma_t)$ defined by \eqref{func}  under flow \eqref{flow}.
\begin{thm}\label{thm1} Functional $Q(\Sigma_t)$ is non-increasing under  the flow \eqref{flow}, provided that the initial surface is horospherical convex.
\end{thm}
\begin{proof}
By Proposition 2.2, Lemma 3.1 and  Proposition \ref{keyprop.}, we have
\begin{eqnarray}\label{l2evolve}
&&\frac{d}{dt}\int_{\Sigma}\bigg\{\sigma_4-\frac{(n-3)(n-4)}{6}\sigma_2+\frac{(n-1)(n-2)(n-3)(n-4)}{24}\bigg\}\nonumber\\
&\leq&(n-5)\int_{\Sigma}\bigg\{\sigma_4-\frac{(n-3)(n-4)}{6}\sigma_2+\frac{(n-1)(n-2)(n-3)(n-4)}{24}\bigg\}.
\end{eqnarray}
On the other hand, by (\ref{var}) and (\ref{eq22}), we also have
\begin{equation}\label{area}
\frac{d}{dt}|\Sigma_t|=\int_{\Sigma_t}\frac{n-4}{4}\frac{\sigma_3\sigma_1}{\sigma_4}d\mu\geq (n-1)|\Sigma_t|.
\end{equation}
Combining (\ref{l2evolve}) and (\ref{area}) together, we complete the proof.

\end{proof}

\begin{thm}\label{thm2} For any  horospherical convex hypersurface  $\Sigma$ in $\H^n$  with $n>5$, we have
\begin{equation}\label{ineq2} Q(\Sigma)\ge \frac{(n-1)(n-2)(n-3)(n-4)}{24}\omega_{n-1}^{\frac{4}{n-1}}.
\end{equation}
Equality holds if and only if $\Sigma$ is a geodesic sphere.
\end{thm}

\begin{proof} Let $\Sigma(t)$ be a solution of flow (\ref{flow}) obtained by the work of Gerhardt \cite{Gerhardt}.
We have showed that  this flow preserves the horospherical
convexity and non-increases  the functional $Q$ in proposition 2.2 and Theorem \ref{thm1} respectively. Hence, to show \eqref{ineq2} we only need  to show
\begin{equation}\label{eq3.18}
\lim_{t\to \infty} Q(\Sigma_t)\ge  \frac{(n-1)(n-2)(n-3)(n-4)}{24}\omega_{n-1}^{\frac{4}{n-1}}.
\end{equation}
Since $\Sigma$ is a horospherical convex hypersurface in $(\H^n,\bar g)$, it can be  written  as a graph of function $r(\theta)$, $\theta\in \mathbb{S}^{n-1}$. We denote $X(t)$ as graphs $r(t,\theta)$ on $\mathbb{S}^{n-1}$ with the standard metric $\hat g$. We set $\lambda(r)=\sinh(r)$ and we have $\lambda'(r)=\cosh(r)$. It is clear that
$$
(\lambda')^2=(\lambda)^2+1.
$$
We define $\varphi(\theta)=\Phi(r(\theta))$. Here $\Phi$ is a function satisfying
$$
\Phi'=\frac{1}{\lambda}.
$$
We need another function
$$
v=\sqrt{1+|\nabla\varphi|^2_{\hat g}}.
$$
By the result of Gerhardt  \cite{Gerhardt}, we have the following results.\\

\begin{lemm}\label{lemmGerhardt}
$$
\lambda=O(e^{\frac{t}{n-1}}),\qquad |\nabla\varphi|+|\nabla^2\varphi|=O(e^{-\frac{t}{n-1}}).
$$
\end{lemm}

The second fundamental form of $\Sigma$ is written in an orthonormal basis
$$
h^i_j=\frac{\lambda'}{v\lambda}\left(\delta^i_j-\frac{\varphi^i_j}{\lambda'}+\frac{{\varphi^i}{\varphi^l}\varphi_{jl}}{v^2\lambda'}\right).
$$
We have also
$$
\nabla \lambda=\lambda\lambda'\nabla\varphi.
$$
We recall the basic facts
$$
\s_4=\frac{1}{24}\left(s_1^4-6s_1^2s_2+8s_1s_3+3s_2^2-6s_4\right),
$$
$$
\s_2=\frac{1}{2}\left(s_1^2-s_2\right),
$$
where
$$
s_k:=\sum_{i}\kappa_{i}^k.
$$
From above, we have the following expansions.
\begin{equation*}\begin{array}{rcl}
s_1&=& \ds\vs \frac{\lambda'}{v\lambda}\left(n-1-\frac{\triangle \varphi}{\lambda'}+\frac{{\varphi^i}{\varphi^j}\varphi_{ij}}{v^2\lambda'}\right) +O(e^{-\frac{6t}{n-1}}),
\\
s_2&=& \ds \vs (\frac{\lambda'}{v\lambda})^2\left(n-1-\frac{2\triangle \varphi}{\lambda'}+\frac{\varphi^i_l\varphi^l_i}{(\lambda')^2}+\frac{2{\varphi^i}{\varphi^j}\varphi_{ij}}{v^2\lambda'}\right) +O(e^{-\frac{6t}{n-1}}),\\
s_3 &=& \ds\vs (\frac{\lambda'}{v\lambda})^3\left(n-1-\frac{3\triangle \varphi}{\lambda'}+\frac{3\varphi^i_l\varphi^l_i}{(\lambda')^2}+\frac{3{\varphi^i}{\varphi^j}\varphi_{ij}}{v^2\lambda'}\right) +O(e^{-\frac{6t}{n-1}}),
\\
s_4&=&\ds\vs (\frac{\lambda'}{v\lambda})^4\left(n-1-\frac{4\triangle \varphi}{\lambda'}+\frac{6\varphi^i_l\varphi^l_i}{(\lambda')^2}+\frac{4{\varphi^i}{\varphi^j}\varphi_{ij}}{v^2\lambda'}\right) +O(e^{-\frac{6t}{n-1}}).
\end{array}\end{equation*}
These give
$$
\s_2=\frac12(\frac{\lambda'}{v\lambda})^2\left((n-1)(n-2)-\frac{2(n-2)\triangle \varphi}{\lambda'}-\frac{\varphi^i_l\varphi^l_i}{(\lambda')^2}+ (\frac{\triangle \varphi}{\lambda'})^2+\frac{2(n-2){\varphi^i}{\varphi^j}\varphi_{ij}}{v^2\lambda'}\right) +O(e^{-\frac{6t}{n-1}}),
$$
and
$$
\begin{array}{lll}
\s_4=&\ds\frac{1}{24}(\frac{\lambda'}{v\lambda})^4(n-3)(n-4)&\ds\left((n-1)(n-2)-\frac{4(n-2)\triangle \varphi}{\lambda'}-\frac{6\varphi^i_l\varphi^l_i}{(\lambda')^2}\right.\\
&&\ds\left.
+6(\frac{\triangle \varphi}{\lambda'})^2+\frac{4(n-2){\varphi^i}{\varphi^j}\varphi_{ij}}{v^2\lambda'}\right) +O(e^{-\frac{6t}{n-1}}),
\end{array}
$$
which imply
\begin{eqnarray*}
l_2&=&\ds \s_4-\frac{(n-3)(n-4)}{6}\s_2+\frac{(n-1)(n-2)(n-3)(n-4)}{24}\\
&=&\ds \frac{(n-1)(n-2)(n-3)(n-4)}{24}\bigg[(\frac{\lambda'}{v\lambda})^2-1\bigg]^2\\
&&\ds- \frac{(n-2)(n-3)(n-4)}{6}\frac{\triangle \varphi}{\lambda'}\bigg(\frac{\lambda'}{v\lambda}\bigg)^2\bigg[(\frac{\lambda'}{v\lambda})^2-1\bigg]\\
&&\ds+\frac{(n-2)(n-3)(n-4)}{6}\bigg(\frac{\lambda'}{v\lambda}\bigg)^2\bigg[(\frac{\lambda'}{v\lambda})^2-1\bigg]\frac{{\varphi^i}{\varphi^j}\varphi_{ij}}{v^2\lambda'}\\
&&\ds+\frac{(n-3)(n-4)}{4}\left(\frac{\triangle \varphi}{\lambda'}\right)^2\bigg(\frac{\lambda'}{v\lambda}\bigg)^2\bigg[(\frac{\lambda'}{v\lambda})^2-\frac13\bigg]\\
&&\ds- \frac{(n-3)(n-4)}{4} \frac{\varphi^i_l\varphi^l_i}{(\lambda')^2}  \bigg(\frac{\lambda'}{v\lambda}\bigg)^2\bigg[(\frac{\lambda'}{v\lambda})^2-\frac13\bigg] +O(e^{-\frac{6t}{n-1}}).
\end{eqnarray*}
From Lemma \ref{lemmGerhardt} and expressions  of $v,\lambda,\lambda'$, we get
$$
(\frac{\lambda'}{v\lambda})^2-1=\frac{1}{\lambda^2}-|\nabla \varphi|^2+O(e^{-\frac{4t}{n-1}}),
$$
which implies
$$
\begin{array}{lll}
\ds \frac{6l_2}{(n-3)(n-4)}&=&\ds  \frac{(n-1)(n-2)}{4}(\frac{1}{\lambda^2}-|\nabla \varphi|^2)^2-(n-2)\frac{\triangle \varphi}{\lambda}(\frac{1}{\lambda^2}-|\nabla \varphi|^2).\\
&&\ds+\left(\frac{\triangle \varphi}{\lambda}\right)^2-\frac{\varphi^i_l\varphi^l_i}{(\lambda')^2}  +O(e^{-\frac{6t}{n-1}}).
\end{array}
$$
We now define a $2$-tensor
$$
A=-\lambda \nabla^2\varphi -\frac12(\lambda^2 |\nabla\varphi|^2-1)\hat g.
$$
One can check that
$$
 \frac{3l_2}{(n-3)(n-4)}=\lambda^{-4}\s_2({\hat g}^{-1}A)+O(e^{-\frac{6t}{n-1}}).
$$
Recall $\varphi (\theta)=\Phi(r(\theta))$. It is easy to check that
$\lambda_i=\lambda' r_i$. It follows that
$$\varphi_i=\lambda_i/\lambda\lambda' \quad \hbox{ and }
\quad
\varphi_{ij}=\frac{\lambda_{ij}}{\lambda^2}-\frac{2\lambda_i\lambda_j}{\lambda^3}+O(e^{-\frac{3t}{n-1}}).
$$
Hence we have
\begin{equation}\label{b1}
 \frac{3l_2}{(n-3)(n-4)}=\lambda^{-4}\s_2({\hat g}^{-1}(-\frac{\nabla^2\lambda}{\lambda}+\frac{2\nabla\lambda\otimes\nabla\lambda }{\lambda^2} -\frac12(\frac{ |\nabla\lambda|^2}{\lambda^2}-1)\hat g))+O(e^{-\frac{6t}{n-1}}).
\end{equation}
Recall the definition of the Schouten tensor
$$
S_{\hat g}=\frac{1}{n-3}\left(Ric_{\hat g}-\frac{R_{\hat g}}{2(n-2)}{\hat g}\right)=\frac12 \hat g.
$$
Its conformal transformation formula is well-known (see for example \cite{Via})
\begin{equation}\label{b2}
S_{\lambda^2\hat g}=-\frac{\nabla^2\lambda}{\lambda}+\frac{2\nabla\lambda\otimes\nabla\lambda }{\lambda^2} -\frac12\frac{ |\nabla\lambda|^2}{\lambda^2}\hat g+S_{\hat g} =-\frac{\nabla^2\lambda}{\lambda}+\frac{2\nabla\lambda\otimes\nabla\lambda }{\lambda^2} -\frac12\frac{ |\nabla\lambda|^2}{\lambda^2}\hat g+\frac 12 \hat g.
\end{equation}
From \eqref{b1} and \eqref{b2}, we obtain
$$
 \frac{3l_2}{(n-3)(n-4)}=\s_2({\lambda^2\hat g})+O(e^{-\frac{6t}{n-1}}).
$$
Recall that the metric on $\Sigma(t)$ has the following expansion
$$
g=\lambda^2(\hat g+\nabla\varphi\otimes\nabla\varphi)=\lambda^2 \hat g+O(1).
$$
It follows
$$
\sqrt{\det(g)}=\lambda^{n-1}(1+O(e^{-\frac{2t}{n-1}})),
$$
which gives
$$
|\Sigma(t)|=(1+o(1))\int_{\mathbb{S}^{n-1}} \lambda^{n-1}=(1+o(1))vol(\lambda^2\hat g).
$$
Similarly, we have
\begin{equation}\label{eq3.1}
\int  \frac{3l_2}{(n-3)(n-4)}=\int_{\mathbb{S}^{n-1}}\s_2({\lambda^2\hat g}) dvol_{\lambda^2\hat g}+ O(e^{\frac{(n-7)t}{n-1}})=(1+o(1))\int_{\mathbb{S}^{n-1}}\s_2({\lambda^2\hat g}) dvol_{\lambda^2\hat g}.
\end{equation}
Here we have used the fact $\int_{\mathbb{S}^{n-1}}\s_2({\lambda^2\hat g}) dvol_{\lambda^2\hat g}= O(e^{\frac{(n-5)t}{n-1}})$ (in fact, we have $\int_{\mathbb{S}^{n-1}}\s_2({\lambda^2\hat g}) dvol_{\lambda^2\hat g}\ge c|\Sigma(t)|^{\frac{n-5}{n-1}}\ge c_1 e^{\frac{(n-5)t}{n-1}}$ for some $c,c_1>0$ see the proof below). We try to use the generalized Sobolev inequality for the conformal metric $\lambda^2\hat g$, which is presented in Proposition 2.3.  For this purpose, we need to show that $ \lambda^2\hat g\in \Gamma_1^+$.
Our observation is that it is asymptotically true. More precisely, we have the following asymptotic property
$$
\begin{array}{llll}
\sigma_1&=&\ds n-1+\frac{n-1}{2}(\frac{1}{\lambda^2}-|\nabla \varphi|^2)-\frac{\triangle \varphi}{\lambda} +O(e^{-\frac{4t}{n-1}})\\
&=&\ds n-1+\frac{1}{\lambda^2} (\frac{n-1}{2}-\frac{n-5}{2}\frac{|\nabla \lambda|^2}{\lambda^2}-\frac{\triangle \lambda}{\lambda} )+O(e^{-\frac{4t}{n-1}})\\
&=&\ds n-1+\s_1({\lambda^2\hat g})+O(e^{-\frac{4t}{n-1}}).
\end{array}
$$
Recall $\Sigma(t)$ is a  horospherical convex hypersurface. As a consequence, $\sigma_1 \ge n-1$ so that
$$\s_1({\lambda^2\hat g})+O(e^{-\frac{4t}{n-1}})\ge 0.$$
We consider $\tilde \lambda:=\lambda^{1-e^{-\frac{t}{n-1}}}$ and the conformal metric $\tilde \lambda^2\hat g$. We have
$$
\tilde \lambda^2\s_1({\tilde \lambda^2\hat g})=\frac{n-1}{2}e^{-\frac{t}{n-1}}+\frac{n-3}{2}e^{-\frac{t}{n-1}}(1-e^{-\frac{t}{n-1}})\frac{|\nabla  \lambda|^2}{ \lambda^2}+(1-e^{-\frac{t}{n-1}})\lambda^2\s_1({\lambda^2\hat g}).
$$
This yields $\tilde \lambda^2\hat g\in \Gamma_1^+$. From  the Sobolev inequality (\ref{Sk}) for the $\s_2$ operator, we have
\begin{equation}\label{eq3.2}
(vol(\tilde \lambda^2\hat g))^{-\frac{n-5}{n-1}}\int_{\mathbb{S}^{n-1}}\s_2({\tilde \lambda^2\hat g}) dvol_{\tilde \lambda^2\hat g}\ge \frac{(n-1)(n-2)}{8}\omega_{n-1}^{\frac{4}{n-1}}.
\end{equation}
On the other hand, we have
\begin{equation}\label{eq3.3}
(vol(\tilde \lambda^2\hat g))^{-\frac{n-5}{n-1}}\int_{\mathbb{S}^{n-1}}\s_2({\tilde \lambda^2\hat g}) dvol_{\tilde \lambda^2\hat g}=(1+o(1))(vol( \lambda^2\hat g))^{-\frac{n-5}{n-1}}\int_{\mathbb{S}^{n-1}}\s_2({ \lambda^2\hat g}) dvol_{ \lambda^2\hat g},
\end{equation}
since
$$
\lambda^{-e^{-\frac{t}{n-1}}}=1+o(1).
$$
As a consequence of (\ref{eq3.1}), (\ref{eq3.2}) and (\ref{eq3.3}), we deduce
$$
\lim_{t\to +\infty} (vol(\Sigma(t)))^{-\frac{n-5}{n-1}}\int_{\Sigma(t)} l_2\ge \frac{(n-1)(n-2)(n-3)(n-4)}{24}\omega_{n-1}^{\frac{4}{n-1}}.
$$
This proves (\ref{eq3.18}), and hence \eqref{ineq2}. When  \eqref{ineq2} is an equality, then $Q$ is constant along the flow. In this case \eqref{area} is an equality, which implies that equality
 in  the inequality
 \[\frac {n-4}4\frac{\sigma_1\sigma_3}{\sigma_4}\ge n-1,\]
 holds. Therefore,
$\Sigma$ is a geodesic sphere.
\end{proof}
Theorem \ref{thm2} has its own interest. It is in fact a Sobolev type inequality. See similar Sobolev type inequalities in \cite{GuanWang}. Now we can finish the proof of our main result.

\

\noindent{\it Proof of Theorem 1.1.} First of all, it is easy to check that for geodesic spheres all inequalities considered in this paper are equalities.

In view of (\ref{func}) and (\ref{ineq2}), we have when $n>5$,
\begin{equation}
\label{equationfinal}
\int_{\Sigma} l_2\ge \frac{(n-1)(n-2)(n-3)(n-4)}{24}\omega_{n-1}^{\frac{4}{n-1}}(|\Sigma|)^{\frac{n-5}{n-1}}.
\end{equation}
Since we can write $\sigma_4=l_2+\frac{(n-2)(n-3)}{6}\sigma_2-\frac{(n-1)(n-2)(n-3)(n-4)}{24}$, combining above with (\ref{eq05}) we get 
\begin{eqnarray}\label{eq_AF2}
\ds \int_\Sigma \s_4 &\ge&\ds\vs  C_{n-1}^4\omega_{n-1}^{\frac{4}{n-1}}(|\Sigma|)^{\frac{n-5}{n-1}} +
\int\frac{(n-2)(n-3)}{6}\sigma_2-\frac{(n-1)(n-2)(n-3)(n-4)}{24}\\
&\ge & \ds\vs C_{n-1}^4\omega_{n-1}\left\{ \left( \frac{|\Sigma|}{\omega_{n-1}} \right)^\frac 12 +
\left( \frac{|\Sigma|}{\omega_{n-1}} \right)^{\frac 12\frac {n-5}{n-1}} \right\}^2.\label{eq_AF2_2}
\end{eqnarray}
This is inequality \eqref{AF2}. By Theorem \ref{thm2}, equality holds if and only if $\Sigma$ is a geodesic sphere,.

When $n=5$, the Euler characteristic is equal to $1$ since the hypersurface $\Sigma$ is star-shaped.
By Remark \ref{Remark1}, we know that \eqref{equationfinal} is an equality when $n=5$, even for any hypersurface diffeomorhpic to a sphere. Hence  in this case, we also have  the above inequalities
with equality in \eqref{eq_AF2}, and hence \eqref{AF2}. Equality in \eqref{AF2} implies equality in \eqref{eq_AF2_2}, which, in turn, implies by \cite{LWX} that $\Sigma$
is a geodesic sphere.
\qed

\vspace{5mm}

\noindent{\it Acknowledgment.} We would like to thank Pengfei Guan for his helpful discussions and constant support and
 Wei Wang for his elegant  proof of Proposition \ref{keyprop.}.


\begin{thebibliography}{99}

\bibitem{AF1} A.D. Alexandrov, {\it Zur Theorie der gemischten Volumina von konvexen K\"orpern, II. Neue Ungleichungen zwischen den gemischten Volumina und ihre Anwendungen}, Mat. Sb. (N.S.) 2 (1937) 1205--1238 (in Russian).
\bibitem{AF2} A.D. Alexandrov, {\it Zur Theorie der gemischten Volumina von konvexen K\"orpern, III. Die Erweiterung zweeier Lehrsatze Minkowskis \"uber die konvexen Polyeder auf beliebige konvexe Flachen}, Mat. Sb. (N.S.) 3 (1938) 27--46 (in Russian).

\bibitem{An} B. Andrews, {\em Pinching estimates and motion of hypersurfaces by curvature functions,} J. reine angew. Math. \textbf{608} (2007) 17--31.
\bibitem{Beckner} W. Beckner, {\it Sharp Sobolev inequalities on the sphere and the Moser-Trudinger inequality,}
Ann. of Math. \textbf{138}(1993), 213--242.

\bibitem{BM} A. A. Borisenko and V. Miquel, {\it Total curvatures of convex hypersurfaces in hyperbolic
space}, Illinois J. Math., 43(1):61--78, (1999).

\bibitem{B} S. Brendle, {\it Constant mean curvature surfaces in warped product manifolds,} \textbf{arXiv:1105.4273}, to appear in  Publ. Math. IHES.

 \bibitem{BE} S. Brendle, M.   Eichmair, {\it Isoperimetric and Weingarten surfaces in the Schwarzschild manifold,} {\bf arXiv:1208.3988}, to appear in JDG.

\bibitem{BHW} S. Brendle, P.-K. Hung, and M.-T. Wang, {\it A Minkowski-type inequality for hypersurfaces in the Anti-deSitter-Schwarzschild manifold},\textbf{ arXiv: 1209.0669}.

\bibitem{BuragoZalgaller}
Y.D. Burago and V.A. Zalgaller, {\it Geometric Inequalities}, Springer, Berlin, (1988).

\bibitem{CM} Esther Cabezas-Rivas, Vicente Miquel, {\it Volume preserving mean curvature flow in the hyperbolic space}, Indiana Univ. Math. J. 56 No. 5 (2007), 2061--2086.


\bibitem{ChangWang} S.-Y. A. Chang and Y.Wang, {\it On Aleksandrov-Fenchel inequalities for k-convex domains,}  Milan J.
Math., {\bf 79} (2011), no. 1, 13--38.

\bibitem{ChangYang}
S.-Y. A. Chang and P. C. Yang, {\it The inequality of Moser and Trudinger and applications to conformal geometry}, Comm. Pure Appl. Math. 56 (2003), 1135--1150.

\bibitem{BN} B. Chow, P. Lu, L. Ni,
{\it  Hamilton's Ricci flow,}  Graduate Studies in Mathematics, 77. American Mathematical Society, Providence, RI; Science Press, New York, (2006).



\bibitem{DGS}  M. Dahl, R. Gicquaud, A. Sakovich, {\it Penrose type inequalities for asymptotically hyperbolic graphs,} \textbf{arXiv:1201.3321}.

\bibitem{deLG} L.L. de lima and F. Gir\~{a}o, {\it An Alexandrov-Fenchel-type inequality in hyperbolic space with an application to a Penrose inequality}, \textbf{arXiv:1209.0438v2}.

\bibitem{Ding} Q. Ding, {\it The inverse mean curvature flow in rotationally symmetric spaces}, Chinese
Annals of Mathematics, Series B, 1--18 (2010).

\bibitem{Fillastre} F. Fillastre, {\it Fuchsian convex bodies: basics of Brunn--Minkowski theory},  \textbf{arXiv:1112.5353}, (2011).



\bibitem{GS} E. Gallego and G. Solanes, {\it Integral geometry and geometric inequalities in hyperbolic
space}, Differential Geom. Appl. \textbf{22}(2005), 315--325.

\bibitem{GeWang1} Y. Ge and  G. Wang, {\it On a conformal quotient equation. II,}
communications in analysis and geometry 21 (2012), 138.

\bibitem{GWW} Y. Ge, G. Wang and J. Wu, {\it A new mass for asymptotically flat manifolds,} \textbf{arXiv:1211.3645}.

\bibitem{GWW2} Y. Ge, G. Wang and J. Wu, {\it The GBC mass for asymptotic hyperbolic manifolds,} in preparation.

\bibitem{GWW_AF_k} Y. Ge, G. Wang, J. Wu, {\it Hyperbolic Alexandrov-Fenchel quermassintegral inequalities II}, in preparation.

\bibitem{Gerhardt} C. Gerhardt, {\it Inverse curvature flows in hyperbolic space}, J. Differential Geom. \textbf{89} (2011), no. 3, 487--527.
\bibitem{Guan} P. Guan, {\it Topics in Geometric Fully Nonlinear Equations,} Lecture Notes, http://www.math.mcgill.ca/guan/notes.html.

\bibitem{GuanLi} P. Guan and J. Li, {\it The quermassintegral inequalities for k-convex starshaped domains,} Adv. Math. \textbf{221}(2009), 1725--1732.

\bibitem{GuanLinWang} P. Guan, C.S. Lin and G. Wang,  {\it TApplication of The Method of Moving Planes to Conformally Invariant Equations},
Mathematische Zeitschrift  247 (2004), 1-19.

\bibitem{GuanMTZ} P. Guan, X. Ma, N. Trudinger, X.Zhu, {\it A form of Alexandrov-Fenchel inequality},
Pure Appl. Math. Q. 6 (2010), no. 4, Special Issue: In honor of Joseph J. Kohn. Part 2, 999--1012.

\bibitem{GuanWang0} P. Guan and G. Wang, {\it A fully nonlinear conformal flow on locally conformally flat manifolds}, J. Reine Angew. Math. 557 (2003), 219--238.


\bibitem{GuanWang}  P. Guan and G. Wang, {\it Geometric inequalities on locally conformally flat manifolds}, Duke Math. J. 124 (2004), 177--212.

\bibitem{HLP} G.H. Hardy, J.E. Littlewood, G. Polya, {\it Inequalities}, Cambridge Univ. Press, Cambridge, (1934).

\bibitem{Huisken} G.Huisken, in preparation. See also \cite{GuanLi}.

\bibitem{KM} K.K. Kwong, P. Miao, {\it A New Monotone Quantity along the Inverse Mean Curvature Flow in $\mathbb R^n$,}
\textbf{arXiv:1212.1906}.
\bibitem{LiHaizhong} H. Li, a private communication.

\bibitem{LWX} H. Li, Y. Wei and C. Xiong, {\it A geometric ineqality on hypersurface in hyperbolic space},{\bf ArXiv:1211.4109}.
\bibitem{Makowski}  M. Makowski, {\it Mixed volume preserving curvature flows in hyperbolic space}, \textbf{arXiv:1208.1898}.
\bibitem{Neves} A. Neves, {\it Insufficient convergence of inverse mean curvature flow on asymptotically hyperbolic manifolds}, J. Diff. Geom. \textbf{84}(2010), 191--229.

\bibitem{Reilly}  R. Reilly, {\it On the Hessian of a function and the curvatures of its graph}, Michigan Math. J. 20 (1973) 373--383.

\bibitem{Schlenker} I. Rivin and  Jean-Marc Schlenker, {\it  On the Schlafli differential formula,} \textbf{ arXiv:math/0001176}.

\bibitem{Santos} L. Santal\'{o}, {\it Integral geometry and geometric probabolity}, Cambridge Mathematical Library. Cambridge University Press, Cambridge, (2004).

\bibitem{Schmidt} E. Schmidt, {\it Die isoperimetrischen Ungleichungen auf der gew\"ohnlichen Kugel und f\"ur Rotationsk\"orper im n-dimensionalen sph\"arischen Raum. } (German)  Math. Z.  46,  (1940), 743--794.
\bibitem{Schneider} R. Schneider, {\it Convex bodies: The Brunn-Minkowski theory}, Cambridge University, (1993), MR1216521.

\bibitem{Solanes_Thesis} G. Solanes, {\it Integrals de curvatura i geometria integral a l'espai hiperbolic}, Univ. Aut. Barcelona, PhD Thesis, (2003).


\bibitem{Via} J. Viaclovsky, {\em  Conformal geometry, contact geometry, and the calculus of variations,} Duke Math. J.  \textbf{101}  (2000),  no. 2, 283--316.


\end{thebibliography}
\end{document}